\documentclass[11pt,twoside]{article}

\usepackage{amsmath,amssymb,amsthm,url}

\newcommand{\Prim}{\mathrm{Prim}}

\newcommand{\vol}{\mathrm{vol}}

\newcommand{\Tr}{\mathrm{Tr}}
\newcommand{\tr}{\mathrm{tr}}
\newcommand{\li}{\mathrm{li}}

\newcommand{\as}{\quad\text{as}\quad}
\newcommand{\tinf}{\to\infty}
\newcommand{\disp}{\displaystyle}
\newcommand{\bsla}{\backslash}

\newcommand{\cO}{\mathcal{O}}

\newcommand{\cT}{\mathcal{T}}

\newcommand{\cU}{\mathcal{U}}

\newcommand{\bC}{\mathbb{C}}
\newcommand{\bR}{\mathbb{R}}
\newcommand{\bQ}{\mathbb{Q}}
\newcommand{\bZ}{\mathbb{Z}}

\newcommand{\bcT}{\hat{\cT}}
\newcommand{\ccT}{\check{\cT}}

\newcommand{\bX}{\bar{X}}

\newcommand{\noi}{\noindent}
\renewcommand{\Re}{\mathrm{Re}}

\newcommand{\divset}{\hspace{3pt}|\hspace{3pt}}
\newcommand{\bigdivset}{\hspace{3pt}\big|\hspace{3pt}}
\newcommand{\Bigdivset}{\hspace{3pt}\Big|\hspace{3pt}}
\newcommand{\biggdivset}{\hspace{3pt}\bigg|\hspace{3pt}}

\newcommand{\gam}{\gamma}
\newcommand{\Gam}{\Gamma}

\newcommand{\btheta}{\bar{\theta}}

\newcommand{\sr}{\mathrm{SL}_2(\bR)}
\newcommand{\sz}{\mathrm{SL}_2(\bZ)}

\newcommand{\vapt}{\vspace{3pt}}
\newcommand{\vbpt}{\vspace{6pt}}
\newcommand{\vcpt}{\vspace{12pt}}

\newtheorem{thm}{Theorem}[section]
\newtheorem{prop}[thm]{Proposition}
\newtheorem{lem}[thm]{Lemma}

\numberwithin{equation}{section}

\title{Universality theorems of the Selberg zeta functions \\for arithmetic groups}
\author{Yasufumi Hashimoto}
\date{}

\begin{document}
\markboth
{Y. Hashimoto}
{Universality theorems of the Selberg zeta functions}
\pagestyle{myheadings}

\maketitle
\renewcommand{\thefootnote}{}

\begin{abstract}
After Voronin proved the universality theorem of the Riemann zeta function in the 1970s, 
universality theorems have been proposed for various zeta and L-functions. 
Drungilas-Garunk\v{s}tis-Ka\v{c}enas' work at 2013 on the universality theorem of the Selberg zeta function 
for the modular group is one of them and is probably the first universality theorem 
of the zeta function of order greater than one. 
Recently, Mishou (2021) extended it by proving the joint universality theorem 
for the principal congruence subgroups. 
In the present paper, we further extend these works by proving the (joint) universality theorem 
for subgroups of the modular group and co-compact arithmetic groups derived 
from indefinite quaternion algebras, which is available in the region wider than the regions in the previous two works.
\end{abstract}

\section{Introduction}

\subsection{Universality theorem} 

The value distributions of the zeta functions in the critical strips 
are, in general, much more complicated than those in the regions of absolute convergence. 
One of the earliest remarkable works providing insight into the value distribution 
of the Riemann zeta function 
$$\zeta(s)=\prod_{p}(1-p^{-s})^{-1}, \qquad \Re{s}>1$$
was Bohr-Courant's result \cite{Bohr} in the 1910s that
the set $\{\zeta(\sigma+it)\in\bC\}_{t\in\bR}$ for $\frac{1}{2}< \sigma\leq 1$ is dense in $\bC$. 
In the 1970s, Voronin \cite{Voro75} made a breakthrough  
by proving the following theorem, called the {\it universality theorem}.
\begin{thm} (Voronin \cite{Voro75}, 1975)  \label{univRZ}
Let $0<r<1/4$ and suppose that $f(s)$ is a non-vanishing analytic function in the interior of the disc $|s|\leq r$ 
and is continuous up to the boundary of this disc. 
Then, for any $\epsilon>0$, we have
$$
\liminf_{T\tinf}\frac{1}{T}\mu\left\{ \tau \in [0,T] \biggdivset  
\max_{|s|<r} \left|\zeta\left(s+\frac{3}{4}+i\tau\right)-f(s)\right|<\epsilon\right\}>0, 
$$
where $\mu$ is the Lebesgue measure on $\bR$.
\end{thm}
Note that the region $|s|\leq r$ in the theorem above was improved to 
a compact subset $K$ of the strip $\{1/2<\Re{s}<1 \}$ with connected complement \cite{Bagchi,Bagchi2}. 
After that, the universality theorems have been proved for various zeta and $L$-functions 
(see, e.g. \cite{Laur,Matsu,Ste} for the progress of the studies of value distributions and universality theorems). 
At 2013, Drungilas-Garunk\v{s}tis-Ka\v{c}enas \cite{Drun} proved the universality theorem 
of the Selberg zeta function associated with the modular group, 
which is, to the best of our knowledge, the first universality theorem 
of the zeta function of order greater than one.

\subsection{Universality theorem of the Selberg zeta function}

Let $H:=\{x+y\sqrt{-1}\divset x,y\in\bR,y>0\}$ be the upper half plane 
and $\Gam$ a discrete subgroup of $\sr$ with $\vol(\Gam\bsla H)<\infty$. 
Denote by $\Prim(\Gam)$ the set of primitive hyperbolic conjugacy classes 
of $\Gam$ and $N(\gam)$ the square of the larger eigenvalue of $\gam$. 
The Selberg zeta function for $\Gam$ is defined by 
\begin{align*}
Z_{\Gam}(s):=\prod_{\gam\in\Prim(\Gam),n\geq0}
(1-N(\gam)^{-s-n}),\qquad \Re{s}>1.
\end{align*}
It is well known that $Z_{\Gam}(s)$ is analytically continued to the whole complex plane 
as a meromorphic function of order $2$ 
and its singular points in $\{1/2<\Re{s}\leq 1\}$ are only a simple zero at $s=1$
and at most a finite number of zeros on the real line (see, e.g. \cite{He}). 
The following analogue of the prime number theorem, 
called the prime geodesic theorem, holds for any $\eta>0$.
\begin{align}\label{pgt}
\begin{split}
\pi_{\Gam}(x):=&\#\{ \gam\in \Prim(\Gam) \divset N(\gam)<x \}\\
=&\li(x)+\sum_{\frac{1}{2}<\rho<1}\li(x^{\rho})+O_{\eta,\Gam}\left(x^{\frac{3}{4}+\eta}\right), \as x\tinf, 
\end{split}
\end{align}
where $\li(x):=\int_{2}^{x}(\log{t})^{-1}dt$ 
and $\{\frac{1}{2}<\rho<1\}$ is the set of at most a finite number of zeros of $Z_{\Gam}(s)$ on the real line. 
Note that the exponent $3/4$ of the error term of the prime geodesic theorem above
was improved to $\frac{25}{36}$  for the modular group 
and the principal congruence subgroups \cite{SY,BF,Cherubini},
and to $\frac{7}{10}$  for the congruence subgroups of the modular group
and co-compact arithmetic groups derived from indefinite quaternion groups \cite{LRS,Ko}. 
We also note that, if the Lindel\"{o}f conjecture hols for the Dirichlet $L$-function, 
it can be improved to $2/3$ for the modular group and the principal congruence subgroups. 

At 2013, Drungilas-Garunk\v{s}tis-Ka\v{c}enas \cite{Drun} proved the following universality theorem 
of $Z_{\Gam}(s)$ associated with the modular group. 
\begin{thm} (Drungilas-Garunk\v{s}tis-Ka\v{c}enas \cite{Drun}, 2013) \label{univD}
Let $\frac{1}{2}<\alpha<1$ be the exponent of the error term of the prime geodesic theorem for $\sz$, i.e. 
$\alpha$ is the constant satisfying
$$\pi_{\sz}(x)-\li(x)\ll_{\eta} x^{\alpha+\eta}, \as x\tinf$$ for any $\eta>0$,
$K$ a compact subset of the strip $\{ \frac{\alpha+1}{2}<\Re{s}<1 \}$ 
with connected complement, and $f(s)$ a non-vanishing function which is continuous in $K$ and analytic 
in the interior of $K$. Then, for any $\epsilon>0$, we have
\begin{align*}
\liminf_{T\tinf}\frac{1}{T}\mu\left\{
\tau\in [0,T] \Bigdivset \max_{s\in K}\left| Z_{\sz}(s+i\tau)-f(s)\right|<\epsilon
\right\}>0.
\end{align*}
\end{thm}

Note that the range $\{ \frac{\alpha+1}{2}<\Re{s}<1 \}$ 
in the universality theorem above was given by 
Landau's formula for square integrals of the Dirichlet series 
(see Section 233--226 in \cite{Lan}). 
Since $\alpha\leq \frac{25}{36}$ holds for $\Gam=\sz$ \cite{SY,BF}, 
we see that the universality theorem above holds in $\{\frac{61}{72}<\Re{s}<1 \}$. 
We also note that it will be improved to  $\{\frac{5}{6}<\Re{s}<1 \}$ 
if the Lindel\"{o}f conjecture for the Dirichlet $L$-function holds. 
Recently, Mishou \cite{Mishou} extended Theorem \ref{univD} by proving 
the {\it joint universality theorem} 
for the principal congruence subgroups 
$$\bar{\Gam}(N):=\left\{ \gam\in \sz \Bigdivset \gam\equiv \pm\begin{pmatrix}1 & \\ & 1 \end{pmatrix} \bmod{N} \right\}$$ 
of the modular group.
\begin{thm} \label{Mi} (Mishou \cite{Mishou}, 2021)
Let $r\geq1$ be an integer, $N_{0}=1$ and
$N_{1},\dots,N_{r}\geq3$ the integers relatively prime to each other. 
Denote by $\frac{1}{2}<\alpha<1$ the exponent of the prime geodesic theorem for the principal congruence subgroup of $\sz$, i.e. 
$\alpha$ is the constant satisfying 
$$\pi_{\Gam}(x)-\li(x)\ll_{\Gam,\eta} x^{\alpha+\eta}, \as x\tinf$$ 
for any $\eta>0$ and any principal congruence subgroup $\Gam$ of $\sz$. 
For $0\leq j\leq r$, let $K_{j}$ be a compact subset of the strip $\{\frac{\alpha+1}{2}<\Re{s}<1 \}$ 
with connected complement and $f_{j}(s)$ a non-vanishing function which is continuous in $K_{j}$ and analytic  
in the interior of $K_{j}$. Then, for any $\epsilon>0$, we have 
\begin{align*}
\liminf_{T\tinf}\frac{1}{T}\mu\left\{
\tau\in [0,T] \Bigdivset \max_{0\leq j\leq r}\max_{s\in K_{j}}\left| Z_{\bar{\Gam}(N_{j})}(s+i\tau)-f_{j}(s)\right|<\epsilon
\right\}>0.
\end{align*}
\end{thm}
Note that, since $\alpha\leq \frac{25}{36}$ holds for the principal congruence subgroups \cite{Cherubini}, 
the range that the joint universality theorem above holds is also $\{ \frac{61}{72}<\Re{s}<1\}$.

\subsection{Main results}
In the present paper, we study the universality of the Selberg zeta function 
for subgroups of the modular group $\sz$ and 
the quaternion groups defined as follows.

Let $a,b$ be square-free and relatively prime positive integers and  
$B:=\bQ+\bQ\alpha+\bQ\beta+\bQ\alpha\beta$ the quaternion algebra over $\bQ$
with $\alpha^2=a$, $\beta^2=b$, $\alpha\beta=-\beta\alpha$. 
Suppose that $B$ is division and  
fix a maximal order $\mathcal{O}$ of $B$. 
Then the group $\mathcal{O}^1$ consisting 
of $q_0+q_1\alpha+q_2\beta+q_3\alpha\beta\in \cO$ 
with $q_0^2-q_1^2a-q_2^2b+q_3^2ab=1$  
can be identified with 
a co-compact discrete subgroup $\Gam=\Gamma_{\mathcal{O}}$ 
of $\sr$ by the map
\begin{align*}
q_0+q_1\alpha+q_2\beta+q_3\alpha\beta\mapsto &
\begin{pmatrix} 
q_0+q_1\sqrt{a}&q_2\sqrt{b}+q_3\sqrt{ab}\\
q_2\sqrt{b}-q_3\sqrt{ab}&q_0-q_1\sqrt{a}
\end{pmatrix}.
\end{align*} 

The main theorem in the present paper is the following universality theorem of the Selberg zeta function.

\begin{thm}\label{thm1} 
Let $\Gam$ be a (not necessarily congruence) subgroup of the modular group $\sz$
or a co-compact arithmetic group $\Gamma_{\mathcal{O}}$ of finite index, 
$K$ a compact subset of the strip $\{\frac{5}{6}<\Re{s}<1 \}$ 
with connected complement and 
$f(s)$ a non-vanishing function which is continuous in $K$ and analytic 
in the interior of $K$. Then, for any $\epsilon>0$, we have
\begin{align*}
\liminf_{T\tinf}\frac{1}{T}\mu\left\{
\tau\in [0,T] \Bigdivset \max_{s\in K}\left| Z_{\Gam}(s+i\tau)-f(s)\right|<\epsilon
\right\}>0.
\end{align*}
\end{thm} 

Note that the universality theorem above is available  in $\{\frac{5}{6}<\Re{s}<1 \}$ 
without the condition $\alpha\leq \frac{2}{3}$ for the corresponding $\Gam$. 
To improve it, we use the upper bound 
(Lemma 3.1 in \cite{HashSq} and Lemma \ref{lemclass} in this paper) of the number $m_{\Gam}(n)$,  
which is almost the same as the number of $\gam\in \Prim(\Gam)$ with the same $N(\gam)$.
Since its multiplicity is more informative than the error term of the prime geodesic theorem, 
we can obtain a better region than that in the previous work.
 
We also prove the following joint universality theorem 
as an improvement of Theorem \ref{Mi}. 

\begin{thm}\label{thm2} 
Let $r\geq1$ be an integer and $\Gam_{1},\dots,\Gam_{r}$ congruence subgroups of $\sz$.
Suppose that 
\begin{align}\label{condition}
\hat{T}_{j}:=\left\{n\in \Tr{\Gam_{j}} \divset \text{$n\not\in \Tr{\Gam_{i}}$ for $1\leq i\leq j-1$}\right\}\neq \emptyset 
\end{align}
for $1\leq j\leq r$,
where $\Tr{\Gam}:=\{\tr{\gam}>2 \divset \gam\in\Prim(\Gam)\}$. 
Let $K_{j}$ be a compact subset of the strip $\{\frac{5}{6}<\Re{s}<1 \}$ 
with connected complement  
and $f_{j}(s)$ a non-vanishing function continuous in $K_{j}$ and  analytic 
in the interior of $K_{j}$.
Then, for any $\epsilon>0$, we have
\begin{align*}
\liminf_{T\tinf}\frac{1}{T}\mu\left\{
\tau\in [0,T] \Bigdivset \max_{1\leq j\leq r}\max_{s\in K_{j}}\left| Z_{\Gam_{j}}(s+i\tau)-f_{j}(s)\right|<\epsilon
\right\}>0.
\end{align*}
\end{thm} 

It is easy to see that, for an integer $N\geq2$, 
$\Tr{\bar{\Gam}(N)}=\{n\in \bZ \divset n\geq3, n\equiv \pm 2\bmod{N^{2}}\}$ 
and $\Tr{\sz}=\{n\in \bZ \divset n\geq3\}$.
Then $\Gam_{1}=\bar{\Gam}(N_{1})$, $\dots$, $\Gam_{r}=\bar{\Gam}(N_{r})$ 
and $\Gam_{r+1}=\sz$ satisfies the condition \eqref{condition} in Theorem \ref{thm2}. 
This means that Theorem \ref{thm2} is an improvement of Theorem \ref{Mi}. 
Checking whether the condition \eqref{condition} holds 
for given congruence subgroups $\Gam_{1},\dots,\Gam_{r}$ is not difficult 
since the sets $\Tr{\Gam_{1}},\dots,\Tr{\Gam_{r}}$ of traces are described by arithmetic progressions. 
Remark that whether the condition \eqref{condition} holds sometimes depends 
on the numbering of $\Gam_{1},\dots,\Gam_{r}$. 
For example, when $\Gam_{1}=\bar{\Gam}(3)$ and $\Gam_{2}=\sz$, 
it holds $\hat{T}_{1}=\Tr{\Gam_{1}}=\{n\geq 3, n\equiv \pm 2\bmod{9}\}$ 
and $\hat{T}_{2}=\Tr{\Gam_{2}}\bsla \Tr{\Gam_{1}}=\{n\geq 3, n\not\equiv \pm 2\bmod{9}\}$, 
and then the condition \eqref{condition} holds. 
However, when $\Gam_{1}=\sz$ and $\Gam_{2}=\bar{\Gam}(3)$, it holds 
$\hat{T}_{1}=\{n\geq 3\}$ and $\hat{T}_{2}=\emptyset$.
Thus, we should be careful about the numbering of $\Gam_{1},\dots,\Gam_{r}$
when checking whether the joint universality holds by the condition \eqref{condition}.
We also remark that there are pairs of subgroups of $\sz$ 
which have the same trace set \cite{Schmutz,Lake}. 
For example, $\Gam_{1}=\bar{\Gam}_{1}(p^{2})
=\left\{ \gam\equiv \pm\begin{pmatrix}1 & * \\ & 1 \end{pmatrix} \bmod{p^{2}}\right\}$ 
and $\Gam_{2}=\bar{\Gam}(p)$ satisfies 
$\Tr{\Gam_{1}}=\Tr{\Gam_{2}}=\{n\geq 3,n\equiv \pm 2\bmod{p^{2}}\}$. 
In such cases, the condition \eqref{condition} does not hold, 
and we should use a different approach to checking joint universality.

\section{Generalized Dirichlet series} 

Drungilas-Garunk\v{s}tis-Ka\v{c}enas \cite{Drun} proposed several propositions and lemmas
to approximate analytic functions by generalized Dirichlet series. 
In this section, we have stated some of them with minor modifications.

Let $\Lambda=\{\lambda\}$ be a monotone increasing sequence of positive real numbers 
tending to infinity, $\{a_{\lambda}\}_{\lambda\in\Lambda}\subset\bC$ and define
$$N(x):=\sum_{\lambda\in\Lambda, \lambda<x}|a_{\lambda}|.$$
We call that the series 
\begin{align}\label{pack}
\sum_{\lambda\in \Lambda}\frac{a_{\lambda}}{e^{\lambda s}}
\end{align} 
satisfies the {\it packing condition} 
if
\begin{align}\label{packing}
\left|N\left(x\pm\frac{c}{x^{2}}\right)-N(x)\right|\gg_{c,\eta} e^{(1-\eta)x}, \as x\tinf 
\end{align}
holds for any $c>0$ and $\eta>0$. 
The following proposition was given in \cite{Drun} 
to approximate analytic functions by the Dirichlet series associated with 
$\Lambda$ and $\{a_{\lambda}\}$.

\begin{prop} \label{appr} (Proposition 2.3 in \cite{Drun} and Proposition 4 in \cite{Mishou}) 
Let $\Lambda$ and $\{a_{\lambda}\}$ be as above.
Suppose that the Dirichlet series \eqref{pack} satisfies the packing condition \eqref{packing}. 
Let $\frac{1}{2}<\sigma_{1}<\sigma_{2}<1$, $K$ a compact subset of $\{\sigma_{1}<\Re{s}<\sigma_{2}\}$ 
with connected complement and 
$g(s)$ a non-vanishing function which continuous on $K$ and is analytic in the interior of $K$.
Then, for any $Q>0$, there exist a constant 
$Y_{0}>0$ depending on $\sigma_{1},\sigma_{2},K,g,\mu$ 
and a sequence $\{\theta_{\lambda}\}_{\lambda\in \Lambda}\subset [0,1)$ 
satisfying 
\begin{align*}
\max_{s\in K}\left|
g(s)-\sum_{\lambda\in \Lambda, Q<e^{\lambda}\leq Y}\frac{a_{\lambda}e(\theta_{\lambda})}{e^{\lambda s}}
\right| \ll 
\sum_{\lambda\in \Lambda, Q<e^{\lambda}\leq Y}\frac{|a_{\lambda}|^{2}}{e^{2\lambda \sigma_{1}}},
\end{align*}
for any $Y>Y_{0}$,  
where $e(x):=e^{2\pi i x}$ and the implied constant depends only on $\sigma_{1},\sigma_{2},K$ and $\Lambda$. 
\end{prop}

Next, we state the following proposition, which is a minor modification 
of Proposition 2.8 and Lemma 2.9 in \cite{Drun} (see also Proposition 5 in \cite{Mishou}).
\begin{prop}\label{mu}  
Let $\alpha>0$ and $\Lambda\subset \bR_{>0}$ be the sequence 
monotone increasing, tending to infinity and linearly independent over $\bQ$. 
Suppose that the series 
$$\sum_{\lambda\in\Lambda}\frac{|a_{\lambda}|^{2}}{e^{2\lambda \sigma}}$$ 
converges for $\sigma>\alpha$. 
Then the following (1) and (2) hold.

\noi (1) For a given series $\{\theta_{\lambda}\}_{\lambda\in\Lambda}\subset [0,1)$, 
a finite subset $\Lambda_{1}$ of $\Lambda$
and $0<\delta<1/2$, let 
$$
S_{T}=S_{T}(\delta,\Lambda_{1}):=\left\{
\tau\in [0,T] \Bigdivset \left\|-\frac{\tau\lambda}{2\pi}-\theta_{\lambda} \right\|<\delta \quad 
\text{for any $\lambda\in\Lambda_{1}$}
\right\},
$$ 
where $\|x\|$ is the distance from the nearest integer to $x$. 
Then we have 
$$\lim_{T\tinf}\frac{\mu(S_{T})}{T}=(2\delta)^{\#\Lambda_{1}}.$$

\vapt

\noi (2) Let $Q>0$ be a large number, $\Lambda_{2}=\Lambda_{2}(Q)$ a finite subset of $\Lambda$ 
with $\Lambda_{1}\cap \Lambda_{2}=\emptyset$, 
$\disp \min_{\lambda\in\Lambda_{2}}e^{\lambda} \geq Q$   
and $K$ a compact subset of $\{\alpha<\Re{s}<1\}$. 
Denote by $S'_{T}=S'_{T}(K,\Lambda_{2},Q)$ the set of $\tau\in S_{T}$ satisfying 
\begin{align*}
\max_{s\in K}\left|
\sum_{\lambda\in \Lambda_{2}(Q)}\frac{a_{\lambda}}{e^{\lambda (s+i\tau)}}
\right| < 
\left(\sum_{\lambda\in \Lambda_{2}(Q)}\frac{|a_{\lambda}|^{2}}{e^{2\lambda \sigma_{1}}}\right)^{1/4},
\end{align*} 
where $\disp \alpha<\sigma_{1}<\max_{s\in K}\Re{s}$. 
Then, for any $0<\beta<1$, there exists $Q>0$ such that 
\begin{align}\label{S'}
\lim_{T\tinf}\frac{\mu(S'_{T})}{T}>(1-\beta)\lim_{T\tinf}\frac{\mu(S_{T})}{T}.
\end{align}
\end{prop}

We use the following lemmas in the proof of the proposition above.
\begin{lem}\label{lemfunc} (Lemma 2.5 in \cite{Gonek} and Lemma 2.5 in \cite{Drun})
Let $K$ be a compact subset of a bounded rectangle $U$,  
$\disp d:=\min_{z\in \partial{U}}\min_{s\in K}|s-z|$ 
and $f(s)$ an analytic function in $U$. 
Then, for a given $\epsilon>0$, we have 
\begin{align*} 
\iint_{U}|f(s)|^{2}ds\leq \epsilon 
\quad \Rightarrow \quad 
\max_{s\in K}|f(s)|\leq \frac{1}{d}\sqrt{\frac{\epsilon}{\pi}}. 
\end{align*} 
\end{lem}

\begin{lem}\label{ww} (\S 8 of Appendix in \cite{Kara} and Lemma 2.10 in \cite{Drun})
Let $w:\bR\to \bR^{N}$ be a curve and suppose that $\{w(t)\divset t\in\bR\}\subset \bR^{N}$ is 
uniformly distributed modulo $1$ in $\bR$. 
Denote by $D$ a closed and Jordan measurable subset of the unit cube in $\bR^{N}$ 
and by $\Omega$ a family of complex-valued continuous functions on $D$. 
If $\Omega$ is uniformly bounded and equi-continuous, then
$$
\lim_{T\tinf}\frac{1}{T}\int_{0}^{T}f(\{w(t)\})1_{D}(t)dt
=\int_{D}f(x_{1},\dots,x_{N})dx_{1}\cdots x_{N}
$$
uniformly with respect to $f\in \Omega$, 
where $\{w(t)\}:=\left(w_{1}(t)-[w_{1}(t)],\dots,w_{N}(t)-[w_{N}(t)]\right)$
for $w(t)=\left(w_{1}(t),\dots,w_{N}(t)\right)$ and  
$1_{D}(t)=1$ if $w(t)\in D$ modulo $1$ and $1_{D}(t)=0$ otherwise.   
\end{lem}

\vcpt

\noi{\bf Proof of Proposition \ref{mu}.} 
(1) See Lemma 2.9 in \cite{Drun} and Proposition 5 in \cite{Mishou}.

\noi(2) Let $U$ be a bounded rectangle with $K\subset U\subset\{\alpha<\Re{s}<\sigma_{1}(<1)\}$ 
and $\disp d:=\min_{z\in \partial{U}}\min_{s\in K}|s-z|$. 
We study the integral
\begin{align}\label{stsq}
\frac{1}{T}\int_{S_{T}}\int_{U}\left|
\sum_{\lambda\in \Lambda_{2}(Q)}\frac{a_{\lambda}}{e^{\lambda (s+i\tau)}}\
\right|^{2}dsd\tau 
=\int_{U} \frac{1}{T}\int_{S_{T}}\left|
\sum_{\lambda\in \Lambda_{2}(Q)}\frac{a_{\lambda}}{e^{\lambda (s+i\tau)}}\
\right|^{2}d\tau ds  
\end{align}
by using Lemma \ref{ww}. 

Let $N_{1}:=\#\Lambda_{1}$,$N_{2}:=\#\Lambda_{2}$ and $N:=N_{1}+N_{2}$.
Denote by $\Lambda_{1}=\{\lambda_{1},\dots,\lambda_{N_{1}}\}$, 
$\Lambda_{2}=\{\lambda_{N_{1}+1},\dots,\lambda_{N}\}$ 
and put 
$w(t):=\left(\frac{\lambda_{1}}{2\pi}t,\dots,\frac{\lambda_{N}}{2\pi}t \right).
$
For a series $\{\theta_{\lambda}\}_{\lambda\in\Lambda}\subset [0,1)$ 
and $0<\delta<1/2$, let 
\begin{align*}
R_{1}=&\left\{(y_{1},\dots,y_{N_{1}})\in [0,1]^{N_{1}} \bigdivset 
\left\| y_{j}-\theta_{\lambda_{j}}  \right\|<\delta \quad (1\leq j\leq N_{1})
 \right\},\\
R_{2}=&\left\{(y_{1},\dots,y_{N})\in [0,1]^{N} \bigdivset 
\begin{array}{ll}
\left\| y_{j}-\theta_{\lambda_{j}}  \right\|<\delta & (1\leq j\leq N_{1}), \\
\left\| y_{j}-1/2  \right\|<1/2 & (N_{1}+1\leq j\leq N )
\end{array}
 \right\},  
\end{align*}
It is clear that $\mu(R_{1})=\mu(R_{2})=(2\delta)^{N_{1}}$. 
According to Lemma \ref{ww}, we have
\begin{align*}
\lim_{T\tinf}\frac{1}{T}\int_{S_{T}}\left|
\sum_{\lambda\in  \Lambda_{2}(Q)}\frac{a_{\lambda}}{e^{\lambda (s+i\tau)}}\
\right|^{2}d\tau 
=&\lim_{T\tinf}\frac{1}{T}\int_{0}^{T}\left|
\sum_{\lambda\in  \Lambda_{2}(Q)}\frac{a_{\lambda}}{e^{\lambda (s+i\tau)}}
\right|^{2}1_{R_{2}}(\tau) d\tau \\
=&\int_{R_{2}}\left|
\sum_{N_{1}+1\leq j\leq N}\frac{a_{\lambda_{j}}}{e^{\lambda_{j}s}}e\left(\lambda_{j}y_{j}\right) 
\right|^{2} dy_{1}\cdots dy_{N}\\
=&\mu(R_{1})\int_{0}^{1} \cdots \int_{0}^{1}\left|
\sum_{N_{1}+1\leq j\leq N}\frac{a_{\lambda_{j}}}{e^{\lambda_{j}s}}e\left(\lambda_{j}y_{j}\right) 
\right|^{2} dy_{N_{1}+1}\cdots dy_{N}\\
\leq & \mu(R_{1})\sum_{\lambda\in  \Lambda_{2}(Q)}\frac{|a_{\lambda}|^{2}}{e^{2\lambda \sigma_{1}}}
\leq  \mu(R_{1})\sum_{\lambda\in \Lambda, e^{\lambda}\geq Q}\frac{|a_{\lambda}|^{2}}{e^{2\lambda \sigma_{1}}}. 
\end{align*}
Then the integral \eqref{stsq} is bounded by 
\begin{align}\label{ineq}
\frac{1}{T}\int_{S_{T}}\int_{U}\left|
\sum_{\lambda\in  \Lambda_{2}(Q)}\frac{a_{\lambda}}{e^{\lambda (s+i\tau)}}
\right|^{2}dsd\tau
\leq \mu(U)\frac{\mu(S_{T})}{T}\sum_{\lambda\in \Lambda, e^{\lambda}\geq Q}\frac{|a_{\lambda}|^{2}}{e^{2\lambda \sigma_{1}}}
+o(1),\as T\tinf. 
\end{align}
Since the sum $\sum_{e^{\lambda}\geq Q}(...)$  in the inequality above tends to zero as $Q\tinf$, 
we see that, for any $0<\beta<1$, there exists $Q>0$ such that 
\begin{align*}
\mu\left\{
\tau\in S_{T} \biggdivset \int_{U}\left|
\sum_{\lambda\in \Lambda_{2}(Q)}\frac{a_{\lambda}}{e^{\lambda (s+i\tau)}}
\right|^{2}ds<
d^{2}\pi\left(\sum_{\lambda\in \Lambda, e^{\lambda}\geq Q}\frac{|a_{\lambda}|^{2}}{e^{2\lambda \sigma_{1}}}\right)^{1/2}
\right\} >(1-\beta)\mu(S_{T}).
\end{align*}
Thus Lemma \ref{mu} follows immediately from Lemma \ref{lemfunc}.
\qed

\section{Explicit formula for the Selberg zeta function}

In this section, we study the logarithm $\log{Z_{\Gam}(s)}$ of the Selberg zeta function, whose branch is  
chosen such that $\log{Z_{\Gam}(s)}\to 0$ as $\Re{s}\tinf$. 
It is easy to see that  
$$
\log{Z_{\Gam}(s)}=-\sum_{\gam\in\Prim(\Gam),j\geq1 } 
\frac{1}{j(1-N(\gam)^{-j})}N(\gam)^{-js}  
$$
for $\Re{s}>1$. 
For $\frac{1}{2}<\Re{s}\leq 1$, it has the following expression as a sum over $\gam\in \Prim(\Gam)$. 
\begin{lem}\label{exp}
Let $\Gam$ be a discrete subgroup of $\sr$ with $\vol(\Gam\bsla H)<\infty$, 
$x>0$ and $s=\sigma+iT\in \bC$ with $\frac{1}{2}<\sigma <1$, $T\geq1$. 
Set 
$$
\psi_{\Gam,s}(x):=\sum_{\begin{subarray}{c}\gam\in\Prim(\Gam),j\geq1 \\ N(\gam)^j<x\end{subarray}} 
\frac{1}{j(1-N(\gam)^{-j})}
\left(1-\frac{N(\gam)^{j}}{x} \right)N(\gam)^{-js}. 
$$
Then we have
\begin{align*}
\log{Z_{\Gam}(s)}=
&-\psi_{\Gam,s}(x)+O_{\eta,\Gam}\left(T^{-2}x^{1-\sigma}+T^{1+\eta}x^{1/2-\sigma+\eta}\right)
\as T,x\tinf
\end{align*}
for any $\eta>0$.
\end{lem}

\begin{proof}
According to Proposition 2.1 in \cite{HashSq}, we have 
\begin{align*}
\frac{Z'_{\Gam}(s)}{Z_{\Gam}(s)}=
&\sum_{\begin{subarray}{c}\gam\in\Prim(\Gam),j\geq1 \\ N(\gam)^j<x\end{subarray}} 
\frac{\log{N(\gam)}}{1-N(\gam)^{-j}}
\left(1-\frac{N(\gam)^{j}}{x} \right)N(\gam)^{-js} \\
&-\sum_{ \frac{1}{2}<\rho\leq 1 }
\frac{x^{\rho-s}}{(\rho-s)(1+\rho-s)}
+O_{\eta,\Gam}\left(T^{1+\epsilon}x^{\frac{1}{2}-\sigma+\eta}\right), 
\as T,x\tinf
\end{align*}
for any $\eta>0$, 
where $\{\frac{1}{2}<\rho\leq 1\}$ is the finite set of zeros of $Z_{\Gam}(s)$ on the real line.
The logarithm of $Z_{\Gam}(s)$ is given by   
\begin{align*}
\log{Z_{\Gam}(s)}
=&\int_{2+iT}^{\sigma+iT}\frac{Z'_{\Gam}(z)}{Z_{\Gam}(z)}dz
+\log{Z_{\Gam}(2+iT)}\\
=&-\psi_{\Gam,\sigma+iT}+\psi_{\Gam,2+iT}+\log{Z_{\Gam}(2+iT)}+O_{\eta,\Gam}\left(T^{-2}x^{1-\sigma}+T^{1+\eta}x^{1/2-\sigma+\eta}\right).
\end{align*}
Due to the prime geodesic theorem \eqref{pgt}, we have
\begin{align*}
&\left| \psi_{\Gam,2+iT}(x)+\log{Z_{\Gam}(2+iT)}\right|\\
&\leq  x^{-1}\sum_{\begin{subarray}{c}\gam\in\Prim(\Gam),j\geq1 \\ N(\gam)^j< x\end{subarray}} 
\frac{1}{j(1-N(\gam)^{-j})}N(\gam)^{-j}
+\sum_{\begin{subarray}{c}\gam\in\Prim(\Gam),j\geq1 \\ N(\gam)^j\geq x\end{subarray}} 
\frac{1}{j(1-N(\gam)^{-j})}
\left(1-\frac{N(\gam)^{j}}{x} \right)N(\gam)^{-2j}\\ 
&\ll_{\eta,\Gam}  x^{-1+\eta}, \as x\tinf  
\end{align*}
for any $\eta>0$.
We thus obtain Lemma \ref{exp}.
\end{proof}

Next, we study the function $\psi_{\Gam,s}(x)$ in more detail 
when $\Gam$ is a subgroup of the modular group or a quaternion  $\Gam_{\cO}$. 
Let 
$$
\Tr{\Gam}:=\{\tr{\gam} \divset \gam\in\Gam, \tr{\gam}>2\}, \qquad 
m_{\Gam}(n):=\sum_{\begin{subarray}{c} \gam\in\Prim(\Gam),j\geq1 \\ \tr{\gam^{j}}=n \end{subarray}} 
\frac{1}{j}.
$$
Since   
$$N(\gam)^{j/2}=\frac{1}{2}\left(\tr{\gam^{j}}+ \sqrt{(\tr{\gam^{j}})^{2}-4} \right),$$ 
the function $\psi_{\Gam,s}(x)$ 
can be expressed by 
\begin{align*}
\psi_{\Gam,s}(x)=&
 \sum_{\begin{subarray}{c} n\in \Tr{(\Gam)} \\ n<X \end{subarray}} 
m_{\Gam}(n) \Delta(n)\left(1-\frac{\epsilon(n)^{2}}{x}\right)\epsilon(n)^{-2s}, 
\end{align*}
where $X:=x^{1/2}+x^{-1/2}$ and   
\begin{align*}
\epsilon(n):=\frac{1}{2}\left(n+\sqrt{n^{2}-4}\right), \qquad 
\Delta(n):=\frac{1}{1-\epsilon(n)^{-2}}.
\end{align*}
It is easy to see that $\Tr{\left(\sz\right)}=\bZ_{\geq 3}$,  
and $\Tr{\Gam}$ is a non-sparse subset of $\bZ_{\geq 3}$ 
when $\Gam$ is a subgroup of the modular group $\sz$ or a quaternion $\Gam_{\cO}$ of finite index. 
Then, by taking $m_{\Gam}(n)=0$ for $n\not \in \Tr{\Gam}$, we can express $\psi_{\Gam,s}(x)$ by 
\begin{align}\label{phi}
\psi_{\Gam,s}(x)=&
\sum_{\begin{subarray}{c} n\in \bZ \\ 3\leq n<X \end{subarray}} 
m_{\Gam}(n)\Delta(n)\left(1-\frac{\epsilon(n)^{2}}{x}\right)\epsilon(n)^{-2s}. 
\end{align}
When $\Gam=\sz$, $\Gam=\Gam_{\cO}$ or 
$\Gam$ is a congruence subgroup of $\sz$,  
it has been known that 
$m_{\Gam}(n)$ can be written as a sum of the class numbers of primitive indefinite binary quadratic forms in the narrow sense 
\cite{Sa1,AKN,HashLS}. 
It is also known that $m_{\Gam}(n)$ is bounded as follows.

\begin{lem}\label{lemclass} (Lemma 3.1 in \cite{HashSq})
When $\Gam$ is a subgroup of the modular group $\sz$ or a quaternion $\Gam_{\cO}$, 
we have
\begin{align}\label{LSineq}
m_{\Gam}(n)\ll_{\eta,\Gam} n^{1+\eta}, \as n\tinf
\end{align}
for any $\eta>0$.
\end{lem}

\section{Proof of the universality theorem}

In this section, we prove Theorem \ref{thm1} by using the propositions and lemmas stated in \S 2 and \S 3. 

\subsection{Linearly independent subset of $\{\log{\epsilon(n)}\}$} 

As studied in \S 3, we see that 
$\log{Z_{\Gam}(x)}$ is written by the Dirichlet series 
over $\Lambda=\{2\log{\epsilon(n)} \divset n\in\bZ_{\geq 3} \}$.
We now generate a linearly independent subset of $\{\log{\epsilon(n)}\}$ over $\bQ$.
Let $\cT\subset \bZ$ be 
$$
\cT:=\left\{ n\geq 3 \divset \text{$\epsilon(n)\neq \epsilon(n_{0})^{k}$ for any $k\geq 2$, $n_{0}\geq3$}\right\}.
$$
For example, $7,14,18\not\in \cT$  since $\epsilon(7)=\epsilon(3)^{2}$, $\epsilon(14)=\epsilon(4)^{2}$, 
$\epsilon(18)=\epsilon(3)^{3}$, 
and $\{3\leq n\leq 18 \divset n\neq 7,14,18\}\subset \cT$.
The following lemma shows that $\cT$ gives a linearly independent subset of $\{\log{\epsilon(n)}\}$ 
and the integers not in $\cT$ are distributed sparsely.

\begin{lem}\label{lemlin} Let $\cT$ be as above. 
Then the following (1) and (2) hold. \\
(1) The set $\left\{\log{\epsilon(n)}\divset n\in \cT\right\}$
is linearly independent over $\bQ$. \\
(2) For any $\eta>0$, we have
$$\#\left\{ n\in\bZ \divset 3\leq n<x, n\not\in\cT \right\}\ll_{\eta} x^{1/2+\eta},\as x\tinf.$$ 
\end{lem}

\begin{proof}
(1) Assume that $\left\{\log{\epsilon(n)}\divset n\in \cT\right\}$
is not linearly independent over $\bQ$, 
i.e. there exist distinct $n_{1},\dots,n_{N}\in \cT$
and non-zero $k_{1},\dots,k_{N}\in\bZ$ 
such that 
$$
k_{1}\log{\epsilon(n_{1})}+\cdots+k_{N}\log{\epsilon(n_{N})}=0. 
$$ 
According to Lemma 4.1 in \cite{Ru}, we see that 
$\epsilon(n_{1}),\dots,\epsilon(n_{N})$ lie in the same quadratic field, 
which means that there exist a non-square integer $D>0$ and integers $u_{1},\dots,u_{N}\geq1$ such that
$n_{i}^{2}-4=Du_{i}^{2}$ for $1\leq i\leq N$. 
Since $\epsilon(n_{i})=\frac{1}{2}(n_{i}+\sqrt{n_{i}^{2}-4})>1$ 
is a unit of the integer ring $\cO$ of $\bQ(\sqrt{D})$, 
there exists an integer $l_{i}\geq1$ such that 
\begin{align}\label{indep}
\epsilon(n_{i})=\epsilon_{0}(D)^{l_{i}}
\end{align}
for $1\leq i\leq N$, 
where $\epsilon_{0}(D)=\frac{1}{2}\left(n_{0}+u_{0}\sqrt{D}\right)$ is the fundamental unit of $\cO_{D}$. 
Here $(n_{0},u_{0})$ satisfies $n_{0}^{2}-Du_{0}^{2}=4$ 
and then 
$$\epsilon_{0}(D)=\frac{1}{2}\left(n_{0}+\sqrt{n_{0}^{2}-4}\right)=\epsilon(n_{0}).$$
This means that \eqref{indep} contradicts the fact that $n_{1},\dots,n_{N}$ are distinct and are elements of $\cT$.  

\vbpt

\noi (2) For $k\geq2$, let 
$$T_{k}(x):=\left\{n\in\bZ \divset 3\leq n\leq x, \text{$\epsilon(n)=\epsilon(n_{0})^{k}$ for some $n_{0}\geq3$}\right\}.$$
We see that $n\in T_{k}$ satisfies 
\begin{align*}
n=\epsilon(n)+\epsilon(n)^{-1}
=\epsilon(n_{0})^{k}+\epsilon(n_{0})^{-k}
=n_{0}^{k}-O(n_{0}^{k-2})
\end{align*}
for some $n_{0}\geq3$ and then
$$\#T_{k}(x)\ll x^{1/k},\as x\tinf.$$
We thus have 
\begin{align}
\left\{ n\in\bZ \divset 3\leq n\leq x, n\not\in \cT \right\}
\leq \sum_{k\geq2}\#T_{k}(x)\ll_{\eta} x^{1/2+\eta},\as x\tinf
\end{align}
for any $\eta>0$.
\end{proof}

\subsection{Partial Dirichlet series}

In the proof of Theorem \ref{thm1}, 
we will divide $\log{Z_{\Gam}}(s)$ by several partial Dirichlet series. 
For a set $A$ of the integers $n\geq3$  
and a series $\{a_{n}\}_{n\in A}\subset [0,1)$, define 
\begin{align*} 
L_{\Gam}(s;A):=&\sum_{\begin{subarray}{c} n\in  A \end{subarray}}
m_{\Gam}(n) \Delta(n)\epsilon(n)^{-2s},\\
L_{\Gam}(s;A;\{a_{n}\}):=&\sum_{\begin{subarray}{c} n\in A \end{subarray}}
m_{\Gam}(n) \Delta(n)\epsilon(n)^{-2s}e\left(a_{n}\right),
\end{align*}
where $e(x):=e^{2\pi ix}$.
For example, 
\begin{align*}
L_{\Gam}(s;\cT\cap [X,Y);\{a_{n}\})
=&\sum_{\begin{subarray}{c} n\in \cT \\ X\leq n<Y \end{subarray}}
m_{\Gam}(n) \Delta(n)\epsilon(n)^{-2s}e\left(a_{n}\right),\\
L_{\Gam}(s;\cT\cap [Y,\infty))
=&\sum_{\begin{subarray}{c} n\in \cT \\ n\geq Y \end{subarray}}
m_{\Gam}(n) \Delta(n)\epsilon(n)^{-2s}. 
\end{align*}
It is clear that $L_{\Gam}(s;\Tr{\Gam})=L_{\Gam}(s;\bZ_{\geq3})=-\log{Z_{\Gam}}(s)$.  
We now prove the following lemma. 

\begin{lem}\label{lemsq}
Let $\Gam$ be a subgroup of the modular group $\sz$
or a co-compact arithmetic group $\Gamma_{\mathcal{O}}$. 
Then the following (1) and (2) hold. 

\vapt

\noi (1) Let $\bar{\cT}:=\bZ_{\geq3}\bsla \cT$.  The series 
\begin{align}\label{sq}
L_{\Gam}(s;\bar{\cT}):=\sum_{\begin{subarray}{c} n\not\in\cT,n\geq3 \end{subarray}}m_{\Gam}(n)\Delta(n)\epsilon(n)^{-2s}
\end{align}
converges absolutely for $\Re{s}>3/4$.

\vapt

\noi (2) For $\frac{1}{2}<\sigma<1$ and $0<Y< T^{3/2}$, we have 
\begin{align}\label{sqint}
\frac{1}{T}\int_{1}^{T}\left|L_{\Gam}\left(\sigma+it; \cT\cap [Y,\infty)\right)\right|^{2}dt 
\ll_{\eta,\Gam}  Y^{3-4\sigma+\eta}+T^{5-6\sigma+\eta}, \as Y,T\tinf
\end{align}
for any $\eta>0$.
\end{lem}

\begin{proof} (1) Due to Lemma \ref{lemclass} and  (2) of Lemma \ref{lemlin}, we have 
\begin{align*}
\sum_{\begin{subarray}{c} n\not\in\cT, 3\leq n\leq x \end{subarray}}m_{\Gam}(n)\Delta(n)\epsilon(n)^{-2s}
\ll_{\eta}\sum_{\begin{subarray}{c}  n\not\in\cT,3\leq n\leq x \end{subarray}}n^{1-2\Re{s}+\eta}
\ll_{\eta}x^{\max{(0,3/2-2\Re{s}+\eta)}} 
\end{align*}
for any $\eta>0$. 
Then the series \eqref{sq} converges absolutely for $\Re{s}>3/4$.

\vapt

\noi (2) Recall that $X=x^{1/2}+x^{-1/2}$ and suppose that $X>Y$. 
Due to Lemma \ref{exp}, we have
\begin{align*}
L_{\Gam}\left(s; \cT\cap [Y,\infty)\right)
=&-\log{Z_{\Gam}(s)}-L_{\Gam}\left(s; \cT\cap [3,Y)\right)
-L_{\Gam}(s;\bar{\cT})\\
=&\psi_{s}(x)-L_{\Gam}\left(s; \cT\cap [3,Y)\right)
-L_{\Gam}(s;\bar{\cT})
+O_{\eta,\Gam}(T^{-2}x^{1-\sigma}+T^{1+\eta}x^{1/2-\sigma+\eta})\\
=&\sum_{n\in \cT,3\leq n<X} m_{\Gam}(n)\Delta(n,x,Y)\epsilon(n)^{-2s}\\
&+\sum_{n\not\in \cT,n\geq 3} m_{\Gam}(n)\Delta(n,x,Y)\epsilon(n)^{-2s}
+O_{\eta,\Gam}\left(T^{-2}x^{1-\sigma}+T^{1+\eta}x^{1/2-\sigma+\eta}\right)\\
=:&M_{1}+M_{2}+O_{\eta,\Gam}\left(T^{-2}x^{1-\sigma}+T^{1+\eta}x^{1/2-\sigma+\eta}\right),
\end{align*}
where 
$$
\Delta(n,x,Y)= \begin{cases} 
-\Delta(n)\epsilon(n)^{2}x^{-1}, &(n<Y),\\
\disp \Delta(n)\left(1-\frac{\epsilon(n)^{2}}{x}\right), &(Y\leq n<X).
\end{cases}
$$
We can get $M_{2}\ll_{\eta,\Gam} x^{3/4-\sigma+\eta}$ easily from (1). 
Then the square integral of $L_{\Gam}^{(1)}(s;[Y,\infty))$ is given as follows.
\begin{align}\label{LL1}
&\frac{1}{T}\int_{1}^{T}\left|L_{\Gam}(\sigma+it;\cT\cap [Y,\infty))\right|^{2}dt\notag\\
\ll & \sum_{n_{1},n_{2}<X}
m_{\Gam}(n_{1}) m_{\Gam}(n_{2})\Delta(n_{1},x,Y)\Delta(n_{2},x,Y)\epsilon(n_{1})^{-2\sigma}\epsilon(n_{2})^{-2\sigma}
\frac{1}{T}\int_{1}^{T}\left(\frac{\epsilon(n_{1})}{\epsilon(n_{2})} \right)^{2it}dt\notag\\
&+O_{\eta,\Gam}\left(x^{3/2-2\sigma+\eta}+T^{-1}x^{2-2\sigma}+T^{2+\eta}x^{1-2\sigma+\eta}\right).
\end{align}
Divide the double sum in the right hand side by $\sum_{n_{1}=n_{2}}+\sum_{n_{1}\neq n_{2}}=:S_{1}+S_{2}$. 
Due to Lemma \ref{lemclass}, we see that 
\begin{align}
S_{1}&=\sum_{3\leq n <X}m_{\Gam}(n)^{2}\Delta(n,x,Y)^{2}\epsilon(n)^{-4\sigma}\notag \\
&=x^{-2}\sum_{3\leq n <Y}m_{\Gam}(n)^{2}\Delta(n)^{2}\epsilon(n)^{4-4\sigma} 
+\sum_{Y\leq n <X} m_{\Gam}(n)^{2}\Delta(n)^{2}\left(1-\frac{\epsilon(n)^{2}}{x}\right)^{2}\epsilon(n)^{-4\sigma}\notag \\
&\ll_{\eta,\Gam}  x^{-2} Y^{7-4\sigma+\eta}+Y^{3-4\sigma+\eta}. \label{LL2}
\end{align}
Furthermore, since
\begin{align*}
\int_{1}^{T}\left(\frac{\epsilon(n_{1})}{\epsilon(n_{2})} \right)^{2it}dt
\ll\frac{1}{\left|\log{\frac{\epsilon(n_{1})}{\epsilon(n_{2})}}\right|}
\ll \begin{cases} \disp \frac{n_{1}}{n_{2}-n_{1}}, &(n_{1}<n_{2}<2n_{1}),\\
1,& (n_{2}\geq 2n_{1}) ,
\end{cases}.
\end{align*}
we have 
\begin{align}
S_{2}& \ll T^{-1} \sum_{n_{1}<X}m_{\Gam}(n_{1}) \Delta(n_{1},x,Y)\epsilon(n_{1})^{-2\sigma}
\sum_{n_{1}<n_{2}<X}m_{\Gam}(n_{2}) \Delta(n_{2},x,Y)\epsilon(n_{2})^{-2\sigma}
\frac{1}{\left|\log{\frac{\epsilon(n_{1})}{\epsilon(n_{2})}}\right|}\notag\\
& \ll_{\eta,\Gam}  T^{-1} \sum_{n_{1}<X}n_{1}^{1-2\sigma+\eta}
\left(\sum_{n_{1}<n_{2}<2n_{1}}n_{2}^{1-2\sigma+\eta}\frac{n_{1}}{n_{2}-n_{1}}
+\sum_{n_{2}\geq 2n_{1}}n_{2}^{1-2\sigma+\eta}\right)\notag \\
&\ll_{\eta,\Gam} T^{-1} x^{2-2\sigma+\eta}. \label{LL3}
\end{align}
Choosing $x=T^{3}$, i.e. $Y<X\sim T^{3/2}$, we can obtain \eqref{sqint} from \eqref{LL1}-\eqref{LL3}. 
\end{proof}

\subsection{Proof of Theorem \ref{thm1}} 

We now give the proof the universality theorem of $\log{Z_{\Gam}(s)}$, 
which is enough to prove the universality theorem of $Z_{\Gam}(s)$ itself 
since 
$$
Z_{\Gam}(s+i\tau)-f(s)=f(s)\left(e^{\log{Z_{\Gam}(s+i\tau)}-\log{f(s)}}-1\right).
$$
Let $X_{1}>0$ and 
divide $\log{f(s)}-\log{Z_{\Gam}(s+i\tau)}$ by 
\begin{align}\label{divuniv}
\begin{split}
\log{f(s)}-\log{Z_{\Gam}(s+i\tau)}
=&\left(\log{f(s)} +L_{\Gam}(s;[3,X_{1})) +L_{\Gam}(s+i\tau;\cT\cap[X_{1},\infty))\right)\\
&+\left(L_{\Gam}(s+i\tau; [3,X_{1}))-L_{\Gam}(s;[3,X_{1})) +L_{\Gam}(s+i\tau;\bar{\cT}\cap [X_{1},\infty))\right).
\end{split}
\end{align}

Due to the prime geodesic theorem \eqref{pgt} and (2) of Lemma \ref{lemlin}, 
we can easily check that
the series 
$$\sum_{\begin{subarray}{c}n\in \cT\end{subarray}}m_{\Gam}(n)\Delta(n)\epsilon(n)^{-2s}$$ 
satisfies the packing condition \eqref{packing}. 
Then, applying Proposition \ref{appr} with 
$$g(x)=\log{f(s)}+L_{\Gam}(s;[3,X_{1})),$$ 
we see that there exist $X_{2}>X_{1}$ and a series $\{\theta_{n}\}_{n\in \cT\cap[X_{1},X_{2}]}\subset [0,1)$ 
such that 
\begin{align} \label{theta}
\begin{split}
&\left| \log{f(s)}+L_{\Gam}(s;[3,X_{1}))
-L_{\Gam}(s;\cT\cap [X_{1},X_{2});\{\theta_{n}\}) \right| \\
\ll &\sum_{\begin{subarray}{c}n\in \cT\\ X_{1}\leq n<X_{2}\end{subarray}}m_{\Gam}(n)^{2}\Delta(n)^{2}\epsilon(n)^{-4\sigma} 
\ll_{\eta} X_{1}^{3-4\sigma+\eta}.
\end{split}
\end{align}
Define the series $\{\btheta_{n}\}_{n\in \cT}$ 
and the subset $S_{T}(\delta)$ of the interval $[0,T]$ by
\begin{align*}
\btheta_{n}:=&\begin{cases} \theta_{n}, & (X_{1}\leq n< X_{2}),\\ 0,& (\text{otherwise}),
\end{cases}\\
S_{T}(\delta):= &\left\{\tau\in[0,T] \Bigdivset \left\| \frac{\tau\log{\epsilon(n)}}{\pi}-\btheta_{n} \right\| <\delta 
\quad \text{for any $n\in \cT\cap[3,X_{2})$} \right\}.
\end{align*}
Then, for $\tau\in S_{T}(\delta)$, we have 
\begin{align}\label{l1}
\begin{split}
&L_{\Gam}(s+i\tau;\cT\cap [X_{1},X_{2}))-L_{\Gam}(s;\cT\cap [X_{1},X_{2});\{\theta_{n}\})\\
& = \sum_{\begin{subarray}{c}n\in\cT \\ X_{1}\leq n<X_{2}\end{subarray}}
m_{\Gam}(n)\Delta(n)\epsilon(n)^{-2s}(\epsilon(n)^{-2i\tau}-e(\theta_{n})) 
\ll_{\eta} \delta X_{2}^{2-2\sigma+\eta}. 
\end{split}
\end{align}
We also have
\begin{align}\label{l2}
&L_{\Gam}(s+i\tau; [3,X_{1}))-L_{\Gam}(s;[3,X_{1})) +L_{\Gam}(s+i\tau;\bar{\cT}\cap [X_{1},\infty))\notag\\
&=\sum_{n\in \bZ,3\leq n<X_{1}}m_{\Gam}(n)\Delta(n)\epsilon(n)^{-2s}(\epsilon(n)^{-2i\tau}-1)
+\sum_{n\not\in \cT,n\geq X_{1}}
m_{\Gam}(n)\Delta(n)\epsilon(n)^{-2s-2i\tau}\notag\\
&\ll_{\eta} \delta X_{1}^{2-2\sigma+\eta}+X_{1}^{\frac{3}{2}-2\sigma+\eta}
\end{align}
for $\tau\in S_{T}(\delta)$ and $\sigma=\Re{s}>3/4$.  

Summarizing \eqref{divuniv}--\eqref{l2}, we see that 
\begin{align}\label{l5}
\begin{split}
&\log{f(s)}-\log{Z_{\Gam}(s+i\tau)}-L_{\Gam}(\sigma+it;\cT\cap [X_{2},\infty)) 
 \ll_{\eta} X_{1}^{\frac{3}{2}-2\sigma+\eta}+\delta  X_{2}^{2-2\sigma+\eta},
\end{split}
\end{align}
if $\tau\in S_{T}(\delta)$ and $\sigma>3/4$.
Now, choose a sufficiently large $X_{3}>X_{2}$ 
and let $S'_{T}(\delta)$ be the set of $\tau\in S_{T}(\delta)$ satisfying 
\begin{align*}
\left| L_{\Gam}(s+i\tau;\cT\cap [X_{2},X_{3}))\right| &<\left(  
\sum_{n\in \cT,n>X_{2}}m_{\Gam}(n)^{2}\Delta(n)^{2}\epsilon(n)^{-4\sigma}\right)^{1/4}
\ll_{\eta} X_{2}^{\frac{3}{4}-\sigma+\eta}< X_{1}^{\frac{3}{4}-\sigma+\eta}.
\end{align*}
Then, if $\tau\in S_{T}'(\delta)$, we see that, for any $\epsilon>0$, 
there exist a sufficiently large $X_{1}>0$ and a sufficiently small $\delta>0$ such that
\begin{align}
\left| \log{f(s)}-\log{Z_{\Gam}(s+i\tau)}-L_{\Gam}(\sigma+it;\cT\cap [X_{3},\infty)) \right|<\frac{1}{2}\epsilon.
\end{align}
Due to Proposition \ref{mu}, we can estimate the measure of $S'_{T}(\delta)$ by 
\begin{align}\label{ST}
\frac{1}{T}\mu\left(S'_{T}(\delta)\right)>\frac{1}{2}(2\delta)^{\#\cT\cap[3,X_{2})}=:\epsilon_{1}.
\end{align}

The remaining part of this proof is to estimate the measure of $\tau$ 
such that $$L_{\Gam}(s+i\tau;\cT\cap[X_{3},\infty))$$ is small enough. 
Let $U$ be a bounded rectangle with $K\subset U \subset\{\frac{5}{6}<\Re{s}<1\}$, 
not including the zeros of $Z_{\Gam}(s)$. 
Put $\disp d:=\min_{z\in \partial{U}}\min_{s\in K} |s-z|$ and $\epsilon_{2}:=\min{(\frac{\epsilon}{2},\epsilon_{1})}>0$. 
According to (2) of Lemma \ref{lemsq}, we can choose sufficiently large $X_{3},T$ such that 
$$
\frac{1}{T}\int_{1}^{T}|L_{\Gam}(\sigma+it;\cT\cap[X_{3},\infty))|^{2}dt<\frac{d^{2}\pi}{2 \mu(U)}\epsilon_{2}^{3}.
$$ 
We then obtain 
\begin{align}
 \mu\left\{\tau\in [0,T] \Bigdivset \max_{s\in K}|L_{\Gam}(s+i\tau;\cT\cap[X_{3},\infty))| <\epsilon_{2}\right\}
 >\left(1-\frac{1}{2}\epsilon_{2}\right)T \label{L1}
\end{align}
from Lemma \ref{lemfunc}. 
The universality theorem  
\begin{align*}
 \mu\left\{\tau\in [0,T] \Bigdivset \max_{s\in K}|\log{f(s)}-\log{Z_{\Gam}(s+i\tau)}| <\epsilon\right\}
 >\frac{1}{2}\epsilon_{2}T
\end{align*}
of $\log{Z_{\Gam}(s)}$ follows from \eqref{ST} and \eqref{L1}. 
\qed

\section{Proof of Theorem \ref{thm2}}

In this section, we prove Theorem \ref{thm2}. 
Before proving it, 
we prepare subsets of $\cT$, 
which are disjoint to each other and 
generate Dirichlet series satisfying the packing condition \eqref{packing}. 

\subsection{Subsets of $\cT$ and partial Dirichlet series}
Recall that $r\geq1$ is an integer, $\Gam_{1},\dots,\Gam_{r}$ are congruence subgroups of $\sz$
and \begin{align*}
\hat{T}_{j}:=\left\{n\in \Tr{\Gam_{j}} \divset \text{$n\not\in \Tr{\Gam_{i}}$ for $1\leq i\leq j-1$}\right\} 
=\Tr{\Gam_{j}}\Big\bsla \bigg( \bigcup_{1\leq i\leq j-1} \Tr{\Gam_{i}} \bigg)
\end{align*}
for $1\leq j\leq r$. 
Put $\cT_{j}:=\cT\cap \Tr{\Gam_{j}}$, 
\begin{align*}
\bcT_{j}:=\cT\cap \hat{T}_{j}=\cT_{j}\Big\bsla \bigg( \bigcup_{1\leq i\leq j-1} \cT_{i} \bigg) ,\qquad 
\ccT_{j}:=\cT_{j}\bsla \bcT_{j}=\cT_{j}\bigcap \bigg( \bigcup_{1\leq i\leq j-1} \cT_{i} \bigg).
\end{align*}
It is easy to see that $\bcT_{i}\cap\bcT_{j}=\emptyset$ for $i\neq j$,  
$\bcT_{j}\cap \ccT_{j}=\emptyset$ and 
\begin{align*}
\bigsqcup_{1\leq i\leq j}\bcT_{i}=\bigcup_{1\leq i\leq j}\cT_{i},
\qquad \bcT_{j}\bigsqcup \ccT_{j}=\cT_{j}.
\end{align*}
We now prove the following lemma.
\begin{lem} \label{partial} Let $\Gam_{1},\dots,\Gam_{r}$ be congruence subgroups of $\sz$.  
If $\hat{T}_{j}\neq \emptyset$, then 
the series 
\begin{align}\label{packlem}
\sum_{n\in \bcT_{j}}m_{\Gam_{j}}(n)\Delta(n)\epsilon(n)^{-2s}
\end{align}
satisfies the packing condition  \eqref{packing}.
\end{lem}

\begin{proof} 
Let $N\geq1$ be the integer such that the principal congruence subgroup $\bar{\Gam}(N)$ 
is a normal subgroup of $\Gam_{1},\dots,\Gam_{r}$ of finite index. 
Denote by $\hat{\Gam}_{j}:=\{g\in\Gam_{j} \divset \text{$\tr{\gam}\not\in \Tr{\Gam}_{i}$ for $1\leq i\leq j-1$}\}$, 
i.e. $\hat{\Gam}_{j}$ is a subset of $\Gam_{j}$ with $\Tr{\hat{\Gam}_{j}}=\hat{T}_{j}$. 
We first prove that $g\bar{\Gam}(N)\subset \hat{\Gam}_{j}$ holds for $g\in \hat{\Gam}_{j}$. 
Assume that it does not hold, namely 
there exist $1\leq i\leq j-1$, $\alpha \in\bar{\Gam}(N)$ and $\beta\in \Gam_{i}$ 
such that $\tr{g\alpha}=\tr{\beta}$. 
Since $\alpha \in\bar{\Gam}(N)$, it holds $\tr{g}\equiv \tr{\beta}\bmod{N}$ 
and then there exists $\alpha_{1} \in\bar{\Gam}(N)$ such that $\tr{g}=\tr{\beta\alpha_{1}}\in \Tr{\Gam_{i}}$. 
This contradicts $g\in \hat{\Gam}_{j}$.

When $\hat{T}_{j}\neq \emptyset$, there exists an element $g\in\hat{\Gam}_{j}$. 
It is clear that $h^{-1}gh\in\hat{\Gam}_{j}$ holds for any $h\in\Gam_{j}$ 
and $h^{-1}gh\bar{\Gam}(N)\subset \hat{\Gam}_{j}$ also holds.  
This means that there exists a conjugacy class $[g]$ of the finite group $\Gam_{j}/\bar{\Gam}(N)$ 
such that  
$$
\sum_{\begin{subarray}{c} n\in \hat{T}_{j} \\  Y_{1}\leq n< Y_{2} \end{subarray}}m_{\Gam_{j}}(n)\Delta(n)
\geq 
 \sum_{\begin{subarray}{c} \gam\in \Prim(\Gam_{j}) \\ 
\gam\subset\hat{\Gam}_{j} \\ \epsilon(Y_{1})^{2}\leq N(\gam)<\epsilon(Y_{2})^{2}\end{subarray}}
\frac{1}{1-N(\gam)^{-1}}
\geq
 \sum_{\begin{subarray}{c} \gam\in \Prim(\Gam_{j}) \\ 
\gam\bar{\Gam}(N)=[g] \\ \epsilon(Y_{1})^{2}\leq N(\gam)<\epsilon(Y_{2})^{2}\end{subarray}}
\frac{1}{1-N(\gam)^{-1}}.
$$
We can check that the series \eqref{packlem} satisfies the packing condition \eqref{packing} 
from the following variant of the prime geodesic theorem, 
called the Chebotarev-type prime geodesic theorem \cite{Sa1,Sunada}. 
$$
\#\{ \gam\in \Prim(\Gam_{j}) \divset \gam\bar{\Gam}(N)= [g], N(\gam)<x \} \\
=C\li(x)+O(x^{a}), \as x\tinf, 
$$
where $0<C\leq 1$ and $1/2<a<1$ are constants 
depending on $\Gam_{j},\bar{\Gam}(N)$ and $g$.  
\end{proof}

\subsection{Proof of Theorem \ref{thm2}}

Let $X_{1}>0$ and 
divide $\log{f_{j}(s)}-\log{Z_{\Gam_{j}}(s+i\tau)}$ by 
\begin{align*}
\log{f_{j}(s)}-\log{Z_{\Gam_{j}}(s+i\tau)}
=&\left(\log{f_{j}(s)}+L_{\Gam_{j}}(s;[3,X_{1}))+ L_{\Gam_{j}}(s+i\tau;\cT\cap [X_{1},\infty)) \right)\\
&+\left(L_{\Gam_{j}}(s+i\tau;[3,X_{1}))-L_{\Gam_{j}}(s;[3,X_{1})) +L_{\Gam_{j}}(s+i\tau;\bar{\cT}\cap [X_{1},\infty))\right).
\end{align*}

First, study the case $j=1$. Since the series 
$$\sum_{\begin{subarray}{c}n\in \cT_{1}\end{subarray}}m_{\Gam_{1}}(n)\Delta(n)\epsilon(n)^{-2s}$$ 
satisfies the packing condition \eqref{packing}, 
applying Proposition \ref{appr} with 
$$g_{1}(s)=\log{f_{1}(s)}+L_{\Gam_{1}}(s;[3,X_{1})),$$ 
we see that there exist $X_{2}^{(1)}>X_{1}$ and 
a series $\{\theta_{n}^{(1)}\}_{n\in \bcT_{j}\cap[X_{1},X_{2}^{(1)})}\subset [0,1)$ 
such that 
\begin{align*}
&\left| \log{f_{1}(s)}+L_{\Gam_{1}}(s;[3,X_{1}])
-L_{\Gam_{1}}(s;\bcT_{1}\cap[X_{1},X_{2}^{(1)});\{\theta_{n}^{(1)}\}) \right|\\
\ll &\sum_{\begin{subarray}{c}n\in \bcT_{1}\\ X_{1}\leq n<X_{2}^{(1)}\end{subarray}}m_{\Gam_{1}}(n)^{2}\Delta(n)^{2}\epsilon(n)^{-4\sigma} 
\ll_{\eta} X_{1}^{3-4\sigma+\eta}.
\end{align*}
Put $\cU_{1}:=\bcT_{1}\cap[X_{1},X_{2}^{(1)})$ and define the series $\{\btheta_{n}^{(1)}\}_{n\in \cT}$ by 
$\btheta_{n}^{(1)}=\theta_{n}^{(1)}$ if $n\in \cU_{1}$  and $\btheta_{n}^{(1)}=0$ otherwise. 

Next, we fix the value $X_{2}^{(j)}>X_{1}$, the series $\{\theta_{n}^{(j)}\}$, $\{\btheta_{n}^{(j)}\}$ 
and the set $\cU_{j}$ for $2\leq j \leq r$ recursively as follows. 
Suppose that $X_{2}^{(i)}$, $\{\theta_{n}^{(i)}\}$, $\{\btheta_{n}^{(i)}\}$ 
and $\cU_{i}$ are already given for $1\leq i \leq j-1$.
Due to Lemma \ref{partial}, we see that 
the series 
$$\sum_{\begin{subarray}{c}n\in \bcT_{j}\end{subarray}}m_{\Gam_{j}}(n)\Delta(n)\epsilon(n)^{-2s}$$ 
satisfies the packing condition \eqref{packing}.
Then, applying Proposition \ref{appr} with 
$$g_{j}(s)=\log{f_{j}(s)}+L_{\Gam_{j}}(s;[3,X_{1}))
-L_{\Gam_{j}}(s;\ccT_{j}\cap\cU_{j-1};\{\btheta_{n}^{(j-1)}\}),$$ 
we see that there exist a value $X_{2}^{(j)}>X_{1}$ and a series 
$\{\theta_{n}^{(j)}\}_{n\in \bcT_{j}\cap[X_{1},X_{2}^{(j)})}\subset [0,1)$ 
such that 
\begin{align*}
&\left| \log{f_{j}(s)}+L_{\Gam_{j}}(s;[3,X_{1}))
-L_{\Gam_{j}}(s;\ccT_{j}\cap\cU_{j-1};\{\btheta_{n}^{(j-1)}\}) 
-L_{\Gam_{j}}(s;\bcT_{j}\cap [X_{1},X_{2}^{(j)});\{\theta_{n}^{(j)}\}) \right|\\
&\ll \sum_{\begin{subarray}{c}n\in \bcT_{j}\\ X_{1}\leq n<X_{2}^{(j)}\end{subarray}}m_{\Gam_{j}}(n)^{2}\Delta(n)^{2}\epsilon(n)^{-4\sigma} 
\ll_{\eta,j} X_{1}^{3-4\sigma+\eta}.
\end{align*} 
Put $\cU_{j}:=\cU_{j-1}\coprod \left(\bcT_{j}\cap [X_{1},X_{2}^{(j)})\right)$  
and define the series $\{\btheta_{n}^{(j)}\}_{n\in \cT}$ by 
$\btheta_{n}^{(j)}=\theta_{n}^{(j)}$ if $n\in \bcT_{j}\cap[X_{1},X_{2}^{(j)})$  and $\btheta_{n}^{(j)}=\btheta_{n}^{(j-1)}$ otherwise. 
Note that $\btheta_{n}^{(r)}=\theta_{n}^{(j)}$ if $n\in \bcT_{j}\cap [X_{1},X_{2}^{(j)})$ for some $1\leq j\leq r$ 
and $\btheta_{n}^{(r)}=0$ otherwise.  

We also put $\{\btheta_{n}\}_{n\in \cT}:=\{\btheta_{n}^{(r)}\}_{n\in \cT}$, 
$\cU:=\cU_{r}\coprod \left(\cT \cap [3,X_{1})\right)$ and  
\begin{align*}
S_{T}=S_{T}(\delta,\cU):=& \left\{\tau\in[0,T] \Bigdivset \left\| \frac{\tau\log{\epsilon(n)}}{\pi}-\btheta_{n} \right\| <\delta 
\quad \text{for any $n\in \cU$} \right\}
\end{align*}
for $\delta>0$ and $T>0$. 
Then, if $\tau\in S_{T}$, we have 
\begin{align}
&L_{\Gam_{j}}\left(s+i\tau; \bcT_{j}\cap [X_{1}, X_{2}^{(j)})\right)
-L_{\Gam_{j}}\left(s; \bcT_{j}\cap [X_{1}, X_{2}^{(j)});\{\btheta_{n}\})\right) \notag \\ 
& = \sum_{\begin{subarray}{c}n\in \bcT_{j}\\ X_{1}\leq n<X_{2}^{(j)}\end{subarray}}m_{\Gam_{j}}(n)\Delta(n)\epsilon(n)^{-2s}
(\epsilon(n)^{-2i\tau}-e(\btheta_{n})) 
\ll_{\eta,j} \delta \left(X_{2}^{(j)}\right)^{2-2\sigma+\eta}, \label{l11} \\ 
&L_{\Gam_{j}}\left(s+i\tau;\ccT_{j}\cap\cU_{j-1}\right)
-L_{\Gam_{j}}\left(s;\ccT_{j}\cap\cU_{j-1}; \{ \btheta_{n} \} \right) \notag \\ 
& = \sum_{\begin{subarray}{c}n\in \ccT_{j}\cap\cU_{j-1}\end{subarray}}m_{\Gam_{j}}(n)\Delta(n)\epsilon(n)^{-2s}
(\epsilon(n)^{-2i\tau}-e(\btheta_{n})) 
\ll_{\eta,j} \delta \left(\bX_{2}^{(j-1)}\right)^{2-2\sigma+\eta},\label{l12}
\end{align}
where $\disp \bX_{2}^{(j-1)}:=\max_{1\leq l\leq j-1}X_{2}^{(l)}$.
We also have  
\begin{align}
&L_{\Gam_{j}}(s+i\tau;[3,X_{1}))-L_{\Gam_{j}}(s;[3,X_{1})) 
+L_{\Gam_{j}}(s+i\tau;\bar{\cT}\cap[X_{1},\infty)) \notag\\
& \ll \sum_{n\in \bZ,3\leq n<X_{1}}m_{\Gam_{j}}(n)\Delta(n)\epsilon(n)^{-2s}| \epsilon(n)^{-2i\tau}-1|
+\sum_{n\not\in \cT,n\geq X_{1}} m_{\Gam_{j}}(n)\Delta(n)\epsilon(n)^{-2\sigma} \notag \\
& \ll_{\eta,j} \delta X_{1}^{2-2\sigma+\eta}+X_{1}^{\frac{3}{2}-2\sigma+\eta}.\label{l31}
\end{align}
if $\sigma>3/4$ and $\tau\in S_{T}$.
Summarizing \eqref{l11}--\eqref{l31}, if $\tau\in S_{T}$, we have 
\begin{align}
&\log{f_{j}(s)}-\log{Z_{\Gam_{j}}(s+i\tau)}-L_{\Gam_{j}}\left(s+i\tau;\bcT_{j}\cap [X_{2}^{(j)},\infty)\right)
-L_{\Gam_{j}}\left(s+i\tau;\ccT_{j}\cap [X_{1},\infty) \bsla \cU_{j-1}\right) \notag \\
&\ll_{\eta,j} \delta \left( \left(X_{2}^{(j)}\right)^{2-2\sigma+\eta} 
+\left(\bX_{2}^{(j-1)}\right)^{2-2\sigma+\eta}+ X_{1}^{2-2\sigma+\eta}\right)
+X_{1}^{\frac{3}{2}-2\sigma+\eta}.\label{l51}
\end{align} 

Now, choose a sufficiently large $\disp X_{3}>\max_{1\leq j\leq r}X_{2}^{(j)}$ and  
let 
\begin{align*}
\bar{\cU}_{j}=\bar{\cU}_{j}(X_{1},X_{3}):=\left( \bcT_{j}\cap [X_{2}^{(j)},X_{3})\right)
\coprod \left( \ccT_{j}\cap [X_{1},X_{3})  \bsla \cU_{j-1}\right).  
\end{align*}
Define the set $S^{(j)}_{T}$ of $\tau\in S_{T}$ satisfying 
\begin{align*}
\left| L_{\Gam_{j}}\left(s+i\tau;\bar{\cU}_{j} \right)\right| 
&< \left(\sum_{n\in \bar{\cU}_{j}}m_{\Gam_{j}}(n)^{2}\Delta(n)^{2}\epsilon(n)^{-4\sigma}
\right)^{1/4} \\ 
& \ll_{\eta,j} \left(  
\sum_{n\in \cT,n>X_{1}}m_{\Gam_{j}}(n)^{2}\Delta(n)^{2}\epsilon(n)^{-4\sigma}
\right)^{1/4} \ll 
X_{1}^{3/4-\sigma+\eta}. 
\end{align*}
This means that, if $\tau\in S^{(j)}_{T}$, for any $\epsilon>0$, 
there exist $X_{1},\delta$ such that
\begin{align}\label{jeps}
\left| \log{f_{j}(s)}-\log{Z_{\Gam_{j}}(s+i\tau)}-L_{\Gam_{j}}(s+i\tau;\cT\cap[X_{3},\infty)) \right|<\frac{1}{2}\epsilon.
\end{align}
Since $\cU,\bar{\cU}_{j}$ are disjoint and  $\disp \min\{n\in \bar{\cU}_{j}\}> X_{1}$, 
 due to (2) of Proposition \ref{mu}, we see that 
\begin{align}\label{Sj}
\mu\left(S_{T}^{(j)}\right)>\left(1-\frac{1}{2r}\right)\mu\left(S_{T}\right) 
\end{align}
if $X_{1}$ is sufficiently large. 
According to \eqref{Sj} and (1) of Proposition \ref{mu}, 
we have 
\begin{align}\label{jj}
\frac{1}{T}\mu\left(\bigcap_{1\leq j\leq r} S^{(j)}_{T}\right)
>\frac{1}{2}\frac{\mu(S_{T})}{T}  
>\frac{1}{2}(2\delta)^{\#\cU}=:\epsilon_{1} 
\end{align}
for sufficiently large $T>0$. 
Thus the set of $\tau\in[0,T]$ satisfying \eqref{jeps} for $1\leq j\leq r$ simultaneously has a positive measure.

The remaining part of this proof is to study $L_{\Gam_{j}}(s+i\tau;\cT\cap[X_{3},\infty))$. 
Let $U$ be a bounded rectangle with $\cup_{1\leq j\leq r}K_{j}\subset U \subset\{\frac{5}{6}<\Re{s}<1\}$, 
not including the zeros of $Z_{\Gam_{j}}(s)$. 
Put $\disp d:=\max_{1\leq j\leq r}\min_{z\in \partial{U}}\min_{s\in K_{j}} |s-z|$ 
and $\epsilon_{2}:=\min{(\frac{\epsilon}{2},\epsilon_{1})}>0$. 
According to Lemma \ref{lemsq}, we see that there exist $X_{3},T$ such that 
$$
\frac{1}{T}\int_{1}^{T}|L_{\Gam_{j}}(\sigma+it;\cT\cap [X_{3},\infty))|^{2}dt<\frac{d^{2}}{2\pi \mu(U)}\epsilon_{2}^{3}
$$
for any $1\leq j\leq r$. 
We then obtain 
\begin{align}
 \mu\left\{\tau\in [0,T] \Bigdivset \max_{1\leq j\leq r}\max_{s\in K}|L_{\Gam_{j}}(s+i\tau;\cT\cap[X_{3},\infty))| <\epsilon_{2}\right\}
 >\left(1-\frac{1}{2}\epsilon_{2}\right)T \label{L11}
\end{align}
from Lemma \ref{lemfunc},
The desired result 
\begin{align}\label{log1}
 \mu\left\{\tau\in [0,T] \Bigdivset \max_{1\leq j\leq r}\max_{s\in K_{j}}|\log{f_{j}(s)}-\log{Z_{\Gam_{j}}(s+i\tau)}| <\epsilon\right\}
 >\frac{1}{2}\epsilon_{2}T
\end{align}
follows from \eqref{jj} and \eqref{L11}. 
\qed

\vcpt

\noi{\bf Acknowledgment.} 
The author was supported by JST CREST no.JPMJCR2113 
and JSPS Grant-in-Aid for Scientific Research (C) no. 22K03234.

\end{document}